\documentclass{article}

\usepackage{authblk}
\usepackage{amsmath,amsthm,amsfonts,graphicx,amssymb}
\usepackage{hyperref}

\title{Separation profiles of free products}
\author{David Hume\thanks{Email address for correspondence: d.hume@bham.ac.uk}}
\affil{University of Birmingham, UK}
\date{\today}                     

\numberwithin{equation}{section}
\newtheorem{theorem}[equation]{Theorem}
\newtheorem{proposition}[equation]{Proposition}
\newtheorem{corollary}[equation]{Corollary}
\newtheorem{lemma}[equation]{Lemma}

\theoremstyle{definition}

\newtheorem{definition}[equation]{Definition}

\newtheorem*{corollary*}{Corollary}
\newtheorem*{theorem*}{Theorem}
\newtheorem*{proposition*}{Proposition}

\newtheoremstyle{citing}
  {3pt}
  {3pt}
  {\itshape}
  {}
  {\bfseries}
  {}
  {.5em}
  {\thmnote{#3}}

\theoremstyle{citing}

\DeclareMathOperator{\cw}{cw}

\DeclareMathOperator{\pw}{pw}
\DeclareMathOperator{\sn}{sn}
\DeclareMathOperator{\tw}{tw}

\DeclareMathOperator{\sumcut}{sc}

\DeclareMathOperator{\cut}{cut}
\DeclareMathOperator{\sep}{sep}

\newcommand{\setcon}[2]{\left\{#1\ \left|\ #2\right.\right\}}

\newcommand{\Cay}{\mathrm{Cay}}

\newcommand{\N}{\mathbb{N}}

\def\XXint#1#2#3{{\setbox0=\hbox{$#1{#2#3}{\int}$}
\vcenter{\hbox{$#2#3$}}\kern-.5\wd0}}

\numberwithin{equation}{section}

\begin{document}
\maketitle

\begin{abstract}
We deduce from a theorem of Dvorak--Norin that the separation and treewidth profiles of graphs are asymptotically equivalent, resolving a question of Huang--Hume--Kelly--Lam. As an application, we calculate the separation profiles of Cayley graphs of tree-graded graphs in terms of their pieces. Examples of tree-graded graphs include Cayley graphs of free products of finitely generated groups.
\end{abstract}

\section{Introduction}
The separation profile, introduced by Benjamini--Schramm--Tim\'ar \cite{BenSchTim-12-separation-graphs}, assigns a sublinear function $\sep_X:\N\to\N$ to each graph $X=(VX,EX)$ that measures how robustly connected finite subgraphs of $X$ are. For a full definition, see $\S\ref{defn:sep}$. It satisfies a remarkable monotonicity property: given two bounded degree graphs $X$ and $Y$, whenever there is a map $\phi:VX\to VY$ and a constant $\kappa$ such that 
\begin{equation}\label{eq:kappareg}
\max_{xx'\in EX} d_Y(\phi(x),\phi(x'))\leq \kappa, \quad \textup{and} \quad \max_{y\in VY}|\phi^{-1}(y)|\leq\kappa    
\end{equation}
there is a constant $C$ such that
\begin{equation}\label{eq:simeq}
    \sep_X(r) \leq C\sep_Y(Cr),\quad \forall r\in \N.
\end{equation}
A map $\phi:VX\to VY$ is said to be \textbf{regular} if it satisfies \eqref{eq:kappareg} for some $\kappa$. Examples of regular maps include: quasi-isometries, coarse embeddings and subgraph inclusions between bounded degree graphs. In particular, if $X$ and $Y$ are Cayley graphs of finitely generated groups $H$ and $G$ (with respect to any choice of finite generating sets), then there is a regular map $VX\to VY$ whenever $H$ is virtually (or even more generally, quasi-isometric to) a subgroup of $G$.

Given functions $f,g:\N\to\N$ we write $f\lesssim g$ if there is a constant $C$ such that $f(r)\leq Cg(Cr)+C$ for all $r\in\N$ and we write $f\simeq g$ if $f\lesssim g$ and $g\lesssim f$ (so \eqref{eq:simeq} states that $\sep_X\lesssim \sep_Y$). We consider separation profiles as $\simeq$-equivalence classes of functions, meaning that a finitely generated group has a well-defined separation profile - that of any of its Cayley graphs.

The separation profile is a powerful invariant in coarse geometry and its applications in group theory. Among finitely generated polycyclic groups it recognises those with polynomial growth of degree $d$ (that satisfy $\sep_G(r)\simeq r^{1-\frac1d}$) from those of exponential growth (that satisfy $\sep_G(r)\simeq r/\log(r)$), and among Riemannian symmetric spaces, it recognises those whose noncompact factor has rank at most 1 from those with higher rank \cite{HumeMackTess-Pprof,HumeMackTess-PprofLie}. It is also connected with many other interesting invariants such as isoperimetry, growth and conformal dimension for compact metric spaces \cite{LeCozGournay,GladShum,HumeMack_confdim}. A long-standing and slightly irritating gap in our understanding of separation profiles of groups concerns its behaviour under free products. We resolve this here.

\begin{theorem}\label{thm:sepfreeprod}
    Let $G$ and $H$ be finitely generated groups. Then
    \[
    \sep_{G*H}(r)\simeq \max\{\sep_G(r),\sep_H(r)\}
    \]
\end{theorem}

Using \cite{PW02} and standard arguments we deduce the following.

\begin{corollary}\label{cor:graphofgroups}
    Let $G$ be the fundamental group of a finite graph of groups with finitely generated vertex groups $\{G_v\}$ and finite edge groups. Then
    \[
    \sep_{G}(r)\simeq \max\{\sep_{G_v}(r)\}
    \]
\end{corollary}

The proof of Theorem \ref{thm:sepfreeprod} has two steps. Firstly, we show how a theorem of Dvorak--Norin implies that the separation profile of any graph can be asymptotically determined by the treewidth (denoted $\tw$, cf.\ Definition \ref{defn:tw}) of its finite subgraphs.

\begin{proposition}\label{prop:twsepequiv}
    Let $X=(VX,EX)$ be a graph. Then for every $r$,
    \[
    \sep_X(r)-1 \leq \max\setcon{\tw(\Gamma)}{\Gamma\leq X,\ |V\Gamma|\leq r} \leq 15\sep_X(r)
    \]
\end{proposition}

Treewidth was introduced by Robertson--Seymour as a measure of how tree-like a graph is \cite{RobSey_tw}. It is an important invariant in its own right and has many applications, see \cite{Bodlaender_treewidth} and references therein.

The function $\tw_X(r)=\max\setcon{\tw(\Gamma)}{\Gamma\leq X,\ |V\Gamma|\leq r}$ is called the treewidth profile in \cite{HHKL26}. Proposition \ref{prop:twsepequiv} resolves Question 4.1 from that paper.

Secondly, we show that the treewidth of a tree-graded graph (a graph made of ``pieces'' arranged in a tree-like manner, cf.\ Definition \ref{defn:tg}) depends only on the treewidth of the pieces.

\begin{proposition}\label{prop:twtg} Let $X$ be a tree-graded graph with pieces $(X_i)_{i\in I}$. Then
\[
    \tw_X(r) = \max\setcon{\tw_{X_i}(r)}{i\in I}.
\]
\end{proposition}

The proof of Theorem \ref{thm:sepfreeprod} then proceeds as follows. Let $S,T$ be finite generating sets of $G$ and $H$ respectively. Then $S\sqcup T$ generates $G*H$ and the Cayley graph $X=\Cay(G*H,S\sqcup T)$ is tree-graded with pieces $(X_i)_{i\in I}$ where each $X_i$ isomorphic to one of $Y=\Cay(G,S)$ or $Z=\Cay(H,T)$. Now we apply Propositions \ref{prop:twsepequiv} and \ref{prop:twtg} to deduce that
\begin{align*}
 \sep_X(r) & \leq \tw_X(r)+1 \\ &\leq \max\{\tw_Y(r),\tw_Z(r)\}+1 \\ & \leq 15\max\{\sep_Y(r),\sep_Z(r)\}+1
\end{align*}
As the separation profile of a finitely generated group does not depend (up to $\simeq$) on the choice of generating set, we are done.

Corollary \ref{cor:graphofgroups} follows, as \cite{PW02} implies that any such $G$ is quasi-isometric to a free product of some of the vertex groups with a (possibly trivial) finite rank free group, and the separation profile of a free group is bounded.

We briefly mention immediate consequences for other invariants introduced in \cite{HHKL26}. In each case, they are defined by replacing treewidth in the definition of the treewidth profile by the mentioned graph layout parameter.

\begin{corollary}\label{cor:graphofgroups2}
    Let $G$ be the fundamental group of a finite graph of groups with finitely generated vertex groups $\{G_v\}$ and finite edge groups. If $G$ is not virtually cyclic, then the \textbf{cutwidth profile} of $G$ is
    \[
        \cw_{G}(r)\simeq \max\{\sep_{G_v}(r),\log(1+r)\}
    \]
    the \textbf{pathwidth profile} of $G$ is
    \[
        \pw_{G}(r)\simeq \max\{\sep_{G_v}(r),\log(1+r)\}
    \]
    and the \textbf{sumcut profile} of $G$ is
    \[
        \sumcut_{G}(r)\simeq r\max\{\sep_{G_v}(r),\log(1+r)\}
    \]
\end{corollary}

\subsection*{Acknowledgements}
The author wishes to thank Joshua Erde for drawing his attention to the work of Dvorak--Norin on which this note relies.

\section{Treewidth, separation and balanced separators}

Here we collect the necessary definitions for non-specialists and deduce Proposition \ref{prop:twsepequiv} from \cite[Theorem 1]{Dvorak-Norin}.

We start with the separation and treewidth profiles.

\begin{definition} Let $\Gamma=(V\Gamma,E\Gamma)$ be a finite graph. The \textbf{cutsize} of $\Gamma$ is the minimum cardinality of a subset $S\subseteq V\Gamma$ with the property that every connected component of $\Gamma-S$ contains at most $\frac12|V\Gamma|$ vertices.    
\end{definition}

\begin{definition}[\cite{BenSchTim-12-separation-graphs}]\label{defn:sep}
    Let $X=(VX,EX)$ be any graph and let $\varepsilon\in (0,1)$. The \textbf{separation profile} of $X$ is the function $\sep_X:\N\to\N$ given by
    \[
        \sep_X(r) = \max\setcon{\cut(\Gamma)}{\Gamma\leq X,\ |V\Gamma|\leq r}
    \]
\end{definition}

\begin{definition}[\cite{RobSey_tw}] \label{defn:tw}
A \textbf{tree-decomposition} of a graph $\Gamma=(V\Gamma,E\Gamma)$ is a pair $(T,\{X_g\}_{g\in VT})$, where $T=(VT,ET)$ is a tree, each $X_g\subseteq V\Gamma$, and
\begin{enumerate}
    \item $\bigcup_{g\in VT} X_g=V\Gamma$,
    \item for every edge $vw\in E\Gamma$ there is some $g\in VT$ such that $v,w\in X_g$,
    \item for each $v\in V\Gamma$, the full subgraph of $T$ with vertex set $\setcon{g\in VT}{v\in X_g}$ is connected.
\end{enumerate}
The width of a tree-decomposition $(T,\{X_g\}_{g\in VT})$ is $\max_{g\in VT} |X_g|-1$, and the \textbf{treewidth} of $\Gamma$, $\tw(\Gamma)$ is the minimum width of a tree-decomposition of $\Gamma$.

The \textbf{treewidth profile} of a graph $X$ is the function $\tw_X:\N\to\N$ given by
\[
    \tw_X(r) = \max\setcon{\tw(\Gamma)}{\Gamma\leq X,\ |V\Gamma|\leq r}
\]
\end{definition}

Dvorak--Norin find an upper bound for the treewidth of any graph in terms of the size of the balanced separators of its subgraphs.

\begin{definition}\label{defn:balancedsep}
    A \textbf{balanced separator} of a finite graph $\Gamma$ is a pair of subsets $A,B\subseteq V\Gamma$ satsifying $|A\setminus B|,|B\setminus A|\leq \frac23|V\Gamma|$ such that there is no edge connecting $A\setminus B$ to $B\setminus A$. The \textbf{size} of a balanced separator is $|A\cap B|$. The \textbf{separation number} of $\Gamma$, $\sn(\Gamma)$ is the smallest $s$ such that every subgraph of $\Gamma$ admits a balanced separator of size at most $s$.
\end{definition}

With all these definitions in place, we may prove Proposition \ref{prop:twsepequiv}.
    
\begin{proposition} Let $X=(VX,EX)$ be a graph. For every $r\in\N$,
\[
    \sep_X(r)-1\leq \tw_X(r) \leq 15\sep_X(r)
\]
\end{proposition}
\begin{proof}
    By the first part of \cite[Lemma 2.3]{BenSchTim-12-separation-graphs}, $\cut(\Gamma)-1\leq \tw(\Gamma)$ for any finite graph $\Gamma$, so taking the maximum over all subgraphs of $X$ with at most $r$ vertices yields the first inequality.
    
    For the second part, by \cite[Theorem 1]{Dvorak-Norin}, $\tw(\Gamma)\leq 15\sn(\Gamma)$ for any finite graph $\Gamma$, so for every $r$
    \[
        \tw_X(r)\leq 15\max\setcon{\sn(\Gamma)}{\Gamma\leq X,\ |V\Gamma|\leq r}
    \]
    Therefore, it suffices to show that every subgraph $\Gamma\leq X$ with at most $r$ vertices admits a balanced separator of size at most $\sep_X(r)$.
    
    For each $\Gamma\leq X$ with at most $r$ vertices, choose a $\frac12$-cutset $C$ of cardinality at most $\sep_X(r)$. Order the connected components $A_1,\ldots,A_k$ of $\Gamma\setminus C$ so that $|VA_i|\geq|VA_j|$ whenever $i\leq j$. Choose the maximum $\ell$ so that $A'=\bigcup_{i=1}^{\ell} VA_i$ satisfies $|A'| \leq \frac23|V\Gamma|$, and set $B'=\bigcup_{i=\ell+1}^k VA_i$. Note that $A'\cap B'=\emptyset$ and $\ell\geq 1$, since $|A_1|\leq\frac12|V\Gamma|$ by assumption. Also, $|A'|\geq\frac{1}{3}|V\Gamma|$ as otherwise $|A_{\ell+1}|\leq |A_\ell|\leq |A'|\leq \frac{1}{3}|V\Gamma|$ and therefore $\bigcup_{i=1}^{\ell+1} |VA_i|\leq\frac23|V\Gamma|$ contradicting the choice of $\ell$. Hence $|B'|\leq\frac23|V\Gamma|$. Define $A=A'\cup C$ and $B=B'\cup C$. It follows that $(A,B)$ is a balanced separator of $\Gamma$ of size $|C|\leq \sep_X(r)$, as required.
\end{proof}

\section{Treewidth of tree-graded graphs}

We start with the definition of a tree-graded graph, which is precisely a graph that is a tree-graded space in the sense of Dru\c{t}u--Sapir \cite{DS05}.

\begin{definition}\label{defn:tg}
    Let $X$ be a graph and let $\{X_i\}_{i\in I}$ be a family of connected subgraphs of $X$. We say $X$ is \textbf{tree-graded} with respect to $\{X_i\}_{i\in I}$ if
    \begin{itemize}
        \item $X=\bigcup_{i\in I} X_i$,
        \item for $i,j\in I$ with $i\neq j$, $X_i\cap X_j$ is either empty or a single vertex,
        \item every simple loop in $X$ is contained in some $X_i$.
    \end{itemize}
    The subgraphs $X_i$ are called pieces.
\end{definition}

Our goal is to prove Proposition \ref{prop:twtg}. We start with a lemma.

\begin{lemma}\label{lem:twjoin}
    Let $\Gamma=(V\Gamma,E\Gamma)$ be a finite graph admitting two connected subgraphs $\Gamma^1,\Gamma^2$ such that $\Gamma=\Gamma^1\cup\Gamma^2$ and $\Gamma^1\cap \Gamma^2$ is either empty or a single vertex. Then $\tw(\Gamma)=\max\{\tw(\Gamma^1),\tw(\Gamma^2)\}$.
\end{lemma}
\begin{proof}
    Let $(T^1,\{X^1_g\}_{g\in VT^1})$ and $(T^2,\{X^2_g\}_{g\in VT^2})$ be tree-decompositions of $\Gamma^1$ and $\Gamma^2$ satisfying $\max\{|X^1_g|\}-1=\tw(\Gamma^1)$ and $\max\{|X^2_g|\}-1=\tw(\Gamma^2)$ respectively.

    Define $T$ to be $T^1\bigsqcup T^2$ plus a single edge $e$ from some $g^1\in VT^1$ to some $g^2\in VT^2$ where $\Gamma_1\cap\Gamma_2\subseteq X^1_{g^1}\cap X^2_{g^2}$. $T$ is clearly a tree. Now, for $g\in VT$, set $X_g=X^i_g$ when $g\in VT^i$. If $\Gamma_1\cap\Gamma_2=\emptyset$ then this is certainly a tree-decomposition, so suppose $\Gamma_1\cap\Gamma_2=\{v\}$. We need only check that the full subgraph of $T$ with vertex set $\setcon{g\in VT}{v\in X_g}$ is connected, but this is by definition the union of a tree in $T^1$ containing $g^1$, a tree in $T^2$ containing $g^2$ and the edge $g^1g^2$ connecting them.
\end{proof}
 
\begin{proposition} Let $X$ be a graph and let $\{X_i\}_{i\in I}$ be a family of subgraphs of $X$. If $X$ is tree-graded with respect to $\{X_i\}_{i\in I}$, then
\[
    \tw_X(r)=\max_i \{\tw_{X_i}(r)\}
\]
\end{proposition}
\begin{proof}
    As each $X_i$ is a subgraph of $X$, the first inequality is automatic. Now fix $r\in \N$, let $\Gamma\leq X$ with $|V\Gamma|\leq r$ and define $\Gamma_i=\Gamma\cap X_i$. We will prove that $\tw(\Gamma)\leq\max\setcon{\tw(\Gamma_i)}{i\in I}$. 
    We may assume that $\Gamma$ is connected, as the treewidth of a graph equals the maximum of the treewidths of its connected components. We may also assume that $|V\Gamma|\geq 2$, as the result clearly holds for a single vertex.

    Set $I_\Gamma= \setcon{i\in I}{|V\Gamma_i|\geq 2}$. Note that $I_\Gamma$ is finite. We will carefully verify that $\Gamma$ is tree-graded with respect to $\{\Gamma_i\}_{i\in I_\Gamma}$. 

    Suppose for a contradiction that some $\Gamma_i$ is not connected. Let $P$ be a minimal length path in $X_i$ connecting two vertices $v,w$ from distinct connected components $\Gamma^0_i$ and $\Gamma^1_i$ of $\Gamma_i$. By minimality $P\cap \Gamma=\{v,w\}$. Let $Q$ be any path from $w$ to $v$ in $\Gamma$. The concatenation $PQ$ is a simple cycle. It is not contained in $X_i$, as there is no path $Q$ connecting $w$ to $v$ in $\Gamma_i=\Gamma\cap X_i$ and it is not contained in any other $X_j$ as $P$ contains an edge from $X_i$ and distinct pieces intersect in at most a single vertex, contradicting the definition of a tree-graded graph.

    Next, we show that $\Gamma=\bigcup_{i\in I_\Gamma} \Gamma_i$. As $X=\bigcup_{i\in I} X_i$, $\Gamma=\bigcup_{i\in I} \Gamma_i$, it suffices to show that whenever $|V\Gamma_i|=1$, there is some $\Gamma_j$ containing $\Gamma_i$ with $|V\Gamma_j|\geq 2$. Suppose $\Gamma_i$ is a single vertex $v$. As $\Gamma$ is connected, and we assumed $|V\Gamma|\geq 2$, $v$ has a neighbour $w$ in $\Gamma$. The edge $vw$ is contained in some $X_j$ and is therefore in $\Gamma_j$. It follows that $\{v\}=\Gamma_i\subset \Gamma_j$ and $|V\Gamma_j|\geq 2$ as required.

    The final two properties are very simple. For any $i,j\in I$, $\Gamma_i\cap\Gamma_j\subseteq X_i\cap X_j$ which is either empty or a single vertex. Any simple loop in $\Gamma$ is contained in some $X_i$, so is also contained in $\Gamma_i$.

    Order $I_\Gamma=\{i_0,\ldots,i_k\}$ as follows. Choose $i_0$. For each $m\geq 1$ we choose $i_m$ so that the subgraphs $\bigcup_{j=0}^{m-1} \Gamma_{i_j}$ and $\Gamma_{i_m}$ intersect. This is always possible as $\Gamma=\bigcup_{i\in I_\Gamma} \Gamma_i$ is connected. We claim that this intersection must be a single vertex. Suppose not, and choose distinct $v,w\in \bigcup_{j=0}^{m-1} \Gamma_{i_j}\cap\Gamma_{i_m}$ at minimal distance in $\Gamma_{i_m}$. Let $P$ be a path from $v$ to $w$ in $\Gamma_{i_m}$ of minimal length. It follows from minimality that $P\cap \bigcup_{j=0}^{m-1} \Gamma_{i_j}=\{v,w\}$. let $Q$ be a path from $w$ to $v$ in $\bigcup_{j=0}^{m-1} \Gamma_{i_j}$ of minimal length. The concatenation $PQ$ is a simple cycle that is not contained in a single $\Gamma_i$, contradicting the definition of a tree-grading.
    
    Now we iteratively apply Lemma \ref{lem:twjoin} to $\Gamma_1=\bigcup_{j=0}^{m-1} \Gamma^{i_j}$ and $\Gamma_2=\Gamma^{i_{m}}$ for each $1\leq m\leq k$.

    It follows that
    \[
        \tw(\Gamma) \leq \max_{i\in I_\Gamma} \{\tw(\Gamma_i)\} \leq \max_{i\in I} \{\tw_{X_i}(r)\}
    \]
    As this holds for all subgraphs $\Gamma\leq X$ with $|V\Gamma|\leq r$ we deduce that
    \[
        \tw_X(r) \leq \max_{i\in I} \{\tw_{X_i}(r)\}  \qedhere
    \]
\end{proof}

\def\cprime{$'$}

\end{document}